\documentclass[10pt,english]{article}

\usepackage{amsmath,amssymb,bm}
\usepackage{ifpdf}
\ifpdf
  \usepackage{graphicx}
\else
  \usepackage[dvipdfmx]{graphicx}
\fi
\usepackage[a4paper,margin=25mm]{geometry}
\usepackage{booktabs,multirow}
\usepackage{subcaption}
\usepackage{makecell}    
\usepackage{longtable}
\usepackage{amsthm}
\usepackage{bbm}
\usepackage{mathdots}  
\usepackage{mathrsfs}
\usepackage{mathtools} 
\usepackage{float}
\usepackage{algpseudocode}
\usepackage{algorithm}
\usepackage{natbib}

\renewcommand{\algorithmicrequire}{\textbf{Input:}}
\renewcommand{\algorithmicensure}{\textbf{Output:}}

\newcommand{\argmin}{\mathop{\rm arg~min}\limits}

\newtheorem{theorem}{Theorem} 

\newtheorem{proposition}[theorem]{Proposition}
\newtheorem{lemma}[theorem]{Lemma}

\makeatletter
  \renewcommand{\eqref}[1]{Eq.~(\ref{#1})}
  \newcommand{\figref}[1]{Figure~\ref{#1}}
  \newcommand{\tabref}[1]{Table~\ref{#1}}
\makeatother



\title{Urban transit network design using spanning tree: A case study of Canberra transit network}
\author{Satoshi Suguira, Kam-Fung Cheung, Michael G. H. Bell, Hitomi Nakanishi, \\ Fumitaka Kurauchi, Supun Perera, Yogi Vidyattama}
\date{}

\begin{document}
\bibpunct{(}{)}{;}{a}{,}{,}
\maketitle

\begin{abstract}
Transit network design plays an important role in public transport. With
the simplicity of spanning tree, this paper adopts the concept of
spanning tree to help (re-)design a public transit network that
addresses passenger utility by minimizing the total
passenger-kilometers, which can be formulated as a mixed-integer
optimization model. However, searching for an optimal minimum
passenger-kilometer spanning tree is intractable for a large network.
This paper proposes a tabu search based heuristic to quickly output a
promising spanning tree. The efficacy of the proposed tabu search based
heuristic is verified using the Canberra bus network data. With the
flexibility of the proposed tabu search heuristic and the efficiency of
the transit system, this paper proposes a greedy algorithm to relax the
number of links constraint to add more links that can further minimize
the total passenger-kilometers. With the case study of Canberra transit
network, this paper provides implications to policy makers to further
improve the design of a public transit network.

\end{abstract}

\section{Introduction}
Transit network design (TND) has been an important topic for
transportation researchers and policy makers for many years
\citet{IbarraRojas2015}. The design of networks that can be operated
cost-effectively by service providers and offer a high level of service
to passengers is desired. To design a transit network to meet certain
requirements, various methodologies have been proposed, including route
setting (e.g., \citet{Mandl1980, Murray2003, Yang2007}),
frequency setting (e.g., \citet{Furth1981, Hadas2012, Verbas2015}), timetabling (e.g., \citet{dePalma2001, Gkiotsalitis2019, Gong2021, Guihaire2010, Huang2023, Li2010, Xie2021}), vehicle scheduling (e.g., \citep{Haghani2002, Kliewer2006, Kulkarni2018, Xi2020, Zhang2021}), driver scheduling (e.g., \citep{Smith1988, Toth2013, Oztop2017}), and their combinations (e.g., \citep{AndradeMichel2021, Baaj1995, Carosi2019, Szeto2011, Sun2019, vanNes1988, Zhen2020}). For an extensive survey of the literature on TND, we refer readers to \citep{Cipriani2012, Guihaire2008, Kepaptsoglou2009, IbarraRojas2015}.

While these methodologies can lead to detailed route plans, policy
makers need to develop a high-level plan, such as a master plan, prior
to producing detailed route plans. The master plan requires a more
comprehensive design of the transit network, such as how to connect the
districts/suburbs in the planning area, where to locate hubs, and where
to place the trunk routes meeting the majority of demand. In this study,
we introduce a transit network design model to policy makers for
developing an optimal network topology as the master plan. The
introduced transit network design model is based on \citet{Bell2020}'s
ferry network model and \citet{Rothlauf2009}'s optimal spanning tree problem.
To have a better scope in this study, detailed planning issues, such as
detailed routes and their frequency, and the location of stations, are
relegated to future studies.

In the transit network design problem, the given depots of vehicles are
helpful in locating hubs in the schematic plan \citep{Guihaire2010, Sun2019}. Some applications of the hub-and-spoke network
design problem to transit are presented in \citet{Alibeyg2016, Huang2018, Wang2012}. For the purpose of finding the
locations of hubs and major trunk routes in public transport, we study
the network design problem using the concept of the spanning tree which
has drawn extensive attentions in the fields of transportation and
operations research \citep{Ahuja1987, Bai2022, Bell2020, Contreras2009, Contreras2010, Hu1974, Kayisoglu2021, Lin2008, Rothlauf2009, Tilk2018, Wei2014, Zetina2019}. A spanning tree is an acyclic network
that has $n - 1$ links, where $n$ is number of nodes in the network.
With this characteristic, the optimum spanning tree is a network, in
which hubs are locations that capture the majority of passenger demands
in the underlying network and address the passenger utility in the
context of urban transit. Therefore, this paper aims to find an optimum
spanning tree locating hubs and trunk routes by minimizing the total
passenger-kilometers.

The transit network design problem (TNDP) with spanning tree constraints
was first studied by \citet{Bell2020}, who proposed a methodology to
find a spanning tree that maximizes entropy for the ferry transit
network in Sydney, Australia. While most TNDPs are formulated as
bi-level optimization problems that passengers aim to maximize their
utility (e.g., minimize their own travel time), service providers aim to
minimize their operating costs. With the concept of spanning tree, a
TNDP can be reduced to a single level optimization problem, as routes
are unique due to the characteristics of the spanning tree. Since there
are $n^{n-2}$ spanning trees in a network of $n$ nodes using
Cayley's formula \citet{Cayley1889}, searching for an optimal solution is
computationally intractable for a large network. \citet{Bell2020}
developed two heuristics to search for a promising spanning tree that
aims to solve the transit network design problem problem.
The transit network design problem (TNDP) with spanning tree constraints
The contributions of the paper are as follows. First, regarding transit
network design we extend \citet{Bell2020} model to a more general
setting by adopting the concept of spanning tree and considering
passenger utility which is measured in passenger-kilometers. The
proposed model aims to locate hubs and trunk routes in the preliminary
planning stage. Second, due to the intractability of searching for an
optimal spanning tree \citep{Bell2020, Hu1974, Kayisoglu2021, Rothlauf2009, Tilk2018, Zetina2019}, we
developed an efficient tabu search heuristic to solve the extended model
for a promising spanning tree. Third, regarding the policy maker's
concern about flexibility in designing routes and achieving smaller
passenger-kilometers, we developed a greedy algorithm by adding
additional links to the spanning tree obtained from the proposed tabu
search. To demonstrate the efficacy of the proposed tabu search
heuristic and the greedy algorithm, we conduct simulations using
Canberra bus network data and compare the performance with the
state-of-the-art heuristics in \citet{Bell2020} in the context of
transit network design. With the findings derived from the Canberra bus
network, this paper provides implications on network design to further
enhance passenger utility.

The remainder of this paper is organized as follows. In Section 2, we
describe the methodology of the transit network design problem using a
spanning tree approach and propose a tabu search heuristic to solve for
a promising spanning tree. In Section 3, we conduct simulations using
the Canberra bus network data to demonstrate the efficacy of the
proposed tabu search. Section 4 provides a greedy algorithm to further
improve the transit network design by relaxing the tree topology
constraint in the proposed model in Section 2, followed by concluding
remarks in Section 5.

\section{Methodology}
\label{sec:methodology}

\subsection{Formulation of transit network design problem using a spanning tree approach}\label{sec:tnd_sta_bilebel_formulation}

Inspired by \citet{Bell2020} entropy maximization model in the context of transit network design, this paper extends \citet{Bell2020} model by using a spanning tree approach.
We formulate the transit network design problem that minimizes the expected passenger-kilometers by searching for a minimum passenger-kilometer spanning tree, which can be formulated as an optimization model. 
We denote the model as Transit Network Design -- Spanning Tree Approach (TND-STA).
Key notations used for the optimization model are shown in \tabref{tab:notations}.

\begin{table}[H]
  \centering
  \caption{Notations for the optimization model Transit Network Design -- Spanning Tree Approach (TND-STA)}
    \begin{tabular}{lp{13cm}}
    \toprule
    Notation & Description \\
    \midrule
    $\mathbb{R}^+$ & Set of non-negative real numbers\\
    $N$ & Set of stations in the network \\
    $W$ & Set of pairs of stations \\
    $\mathcal{P}$ & Set of paths in the network\\
    $\mathcal{P}_w$ & Set of paths that connect station pair $w\in W$\\
    $\mathcal{P}_{w\left(p\right)}$ & Set of paths that connect station pair $w\in W$ at either end of path $p\in\mathcal{P}$\\
    $E$ & Set of links, represented by their start and end nodes, i.e., stations in the transit network, $\left(ij\right)\in E;i,j\in N$ \\
    $\boldsymbol{h}$ & Vector of passenger path flows (in passengers per day), with element $h_p\in\mathbb{R}^+,\ p\in\mathcal{P}$ \\
    $\boldsymbol{B}$ & Origin-destination-path (OD-path) incidence matrix, with element $b_{wp}=1,\ w\in W,\ p\in\mathcal{P}$ if path $p$ connects OD pair $w$; and 0 otherwise \\
     & (Note: $w$ can represent $\left(ij\right)$ only if the direct link $\left(ij\right)$ connects the pair of station $w$, i.e., (station $i$, station $j$), and vice versa.)\\
    $\boldsymbol{d}$ & OD matrix in vector form, with element $d_w\in\mathbb{R}^+,w\in W$ \\
    $\tau$ & Travel distance budget \\
    $M$ & A very large number which can set to infinity \\
    $\boldsymbol{c}$ & Vector of distance (in kilometers) of all pairs of stations in a spanning tree, with element $c_w\in\mathbb{R}^+,w\in W$ \\
    $\boldsymbol{t}$ & Vector of link travel distances (in kilometers), with element $t_{\left(ij\right)}\in\mathbb{R}^+$, where (ij) is the link connecting station $i\in N$ to station $j\in N$ \\
    $\boldsymbol{Q}$ & Link-path incidence matrix, with element $q_{\left(ij\right)p}=1,\ \left(ij\right)\in E,\ p\in\mathcal{P}$ if link ($ij$) lies on path $p$; and 0 otherwise \\
    $\boldsymbol{x}$ & Decision variable -- Vector of link passenger flows (in passengers per day), with element $x_{\left(ij\right)}\in\mathbb{R}^+$, where (ij) is the link connecting station $i\in N$ and station $j\in N$ \\
    $\boldsymbol{y}$ & Decision variable -- Vector of $\left\{0,\ 1\right\}^{\left|N\right|^2}$, with element $y_{\left(ij\right)}=1$ if a direct service exists between stations $i,\ j\in N:i\neq j$; and 0 otherwise \\
    \bottomrule
    \end{tabular}%
  \label{tab:notations}%

\end{table}

The model of the transit network design using the spanning tree approach (TND-STA) is formulated as follows:

\begin{equation}{\label{eq:tnd_sta_obj}}
  \min_{\boldsymbol{x},\ \boldsymbol{y}}\sum_{w\in W}{d_wc_w} \\
\end{equation}
\begin{align}
\text{s.t.}\quad & d_w=\sum_{p\in\mathcal{P}_w}{b_{wp}h_p};\forall w\in W \label{eq:tnd_sta_flow_conservation}\\
& c_w=\sum_{\left(ij\right)\in E}\sum_{p\in\mathcal{P}_w}{b_{wp}q_{\left(ij\right)p}t_{\left(ij\right)}\ };\forall w\in W \label{eq:tnd_sta_distance}\\ 
& x_{\left(ij\right)}=\sum_{w\in W}\sum_{p\in\mathcal{P}_w}{b_{wp}q_{\left(ij\right)p}h_p};\forall i,\ j\in N \label{eq:tnd_sta_link_flow}\\
& \tau\geq\sum_{w\in W}\sum_{p\in\mathcal{P}_w}{c_wb_{wp}h_p}=\sum_{\left(ij\right)\in E}{t_{\left(ij\right)}x_{\left(ij\right)}} \label{eq:tnd_sta_travel_time_budget}\\
& \sum_{i,j\in N} y_{\left(ij\right)}\leq |N|-1 \label{eq:tnd_sta_num_links}\\
& \sum_{i,j\in V} y_{\left(ij\right)}\leq \left|V\right|-1;\forall V\subset N,\ V\neq N,\ V\neq\emptyset \label{eq:tnd_sta_no_cycles} \\
& h_p>0;\forall p\in\mathcal{P} \label{eq:tnd_sta_positive_h}\\
& x_{\left(ij\right)}\le My_{\left(ij\right)};\forall i,\ j\in N \label{eq:tnd_sta_link_capacity}\\
& y_{\left(ij\right)}\in\left\{0,1\right\};\forall i,j\in N \label{eq:tnd_sta_binary_y} 
\end{align}

\eqref{eq:tnd_sta_obj} is the objective of the TND-STA that minimizes the expected passenger-kilometers in the transit network. \eqref{eq:tnd_sta_flow_conservation} is the flow conservation constraint given an OD matrix of trips. Formulae \eqref{eq:tnd_sta_distance} and \eqref{eq:tnd_sta_link_flow} present link-path relationship on the sense of travel distance and passenger flow, respectively. Formula \eqref{eq:tnd_sta_travel_time_budget} represents the travel distance budget constraint. Formula \eqref{eq:tnd_sta_num_links} represents the constraint on the number of links to make the network a spanning tree. Formula (7) is to ensure that there are no cycles in the network and eliminate subtours. Formula \eqref{eq:tnd_sta_positive_h} ensures that there are passengers in each path, while Formula \eqref{eq:tnd_sta_link_capacity} ensures that only operational links are used with M being a very large number. Formula \eqref{eq:tnd_sta_binary_y} restricts variables $y_{\left(ij\right)},i,j\in N$ to be binary. Besides, restricting variables $x_{\left(ij\right)},i,j\in N$ to be non-negative is redundant in the model due to the positivity of the auxiliary variables $h_p,p\in\mathcal{P}$, the non-negativity link-path incidence matrix $\boldsymbol{Q}$, and their relation in \eqref{eq:tnd_sta_link_flow}.

The objective function in Model P3 in \citet{Bell2020} is defined to maximize passenger-utility which is shown below:

\begin{equation}\label{eq:bell_entropy_maximization}
\max{\sum_{w\in W}{d_w\ln{\sum_{p\in\mathcal{P}_w}\exp{\left(-\lambda\sum_{\left(ij\right)\in E}{q_{\left(ij\right)p}t_{ij}}\right)}}}}
\end{equation}

Formula \eqref{eq:bell_entropy_maximization} is equivalent to Formula \eqref{eq:tnd_sta_obj} when the passenger heterogeneity parameter $\lambda$ is set to infinity (\citet{Bell2020}). Furthermore, the proposed tabu search algorithm in \ref{sec:ts_algorithm} can be applied to the entropy maximization problem with a small modification. Although the problem in this paper does not consider the travel time budget constraint on each OD pair, the travel distance budget constraint captures the travel time budget implicitly. If necessary, we can add the travel time budget constraint to the model easily and apply the same tabu search algorithm to solve the extended model, which shows the practicality of the proposed algorithm.

Although searching for a tree with minimum distances between all node pairs is NP-hard \citep{Garey1979} and the complexity of searching a minimum passenger-kilometer spanning tree is unknown, \citet{Bell2020} developed two greedy heuristic algorithms: Link Swapping algorithm and Link Deletion algorithm to obtain maximal passenger-utility spanning trees for the Sydney ferry network. Both Link Swapping and Link Deletion algorithms are computationally demanding for a large network, as they need to calculate the inverse of a matrix. Therefore, with its success on solving various problems in the field of transportation \citep{Chen2022, Hu2018, Xiao2018, Xu2022, Xue2021}, we propose a tabu search based heuristic to obtain a promising solution to TND-STA quickly.

\subsection{Solution algorithm: Link Swapping with Tabu Search}\label{sec:ts_algorithm}

Inspired by \citet{Oncan2008} and \citet{Rothlauf2009} algorithm design strategies on searching for an optimal spanning tree, we develop a tabu search based heuristic, called Link Swapping with Tabu Search, to search for a minimum passenger-kilometer spanning tree. Key notations used for the tabu search heuristic are shown in \tabref{tab:ts_notations}.

\begin{table}[H]
  \centering
  \caption{Notations for the tabu search heuristic Link Swapping with Tabu Search}
    \begin{tabular}{lp{13cm}}
    \toprule
    Notation & Description \\
    \midrule
    $\boldsymbol{c}\left(T\right)$ & Vector of distances along spanning tree $T$ between OD pair $w$, with element $c_w\in\mathbb{R}^+,\ w\in W$ \\
    $Z\left(T\right)$ & Value of the objective function of spanning tree $T$, defined as $Z(T)=\boldsymbol{c}\left(T\right)^T\ \boldsymbol{d}$, where $\boldsymbol{d}$ is the OD matrix in vector form \\
    $m\left(T\right)$ & Number of nodes in tree $T$ \\
    $\mathcal{A}$ & Set of randomly selected links \\
    $\mathcal{L}$ & Tabu list \\
    $T^\ast$ & Optimal spanning tree \\
    $Z^\ast=Z\left(T^\ast\right)$ & Optimal value of the objective function of the optimal spanning tree $T^\ast$ \\
    $_nT$ & Spanning tree in the n-th iteration \\
    ${_n^1}T_{-a} ,{_n^2}T_{-a} $ & Subtrees/Subnetworks 1 and 2 after removing link $a$ in ${_n}T$ \\
    ${_n}T_{-a}^{+b}$ & Spanning tree created by removing link $a$ and inserting link $b$ in ${_n}T$ \\
    \bottomrule
    \end{tabular}%
  \label{tab:ts_notations}%
\end{table}

The procedure of the general tabu search \citet{Glover1989} is as follows:
\begin{enumerate}
  \item Initialize tabu list $\mathcal{L}\leftarrow \emptyset$, current state S, maximum number of iterations $\phi$, current iteration $n\leftarrow0$, best state $S^\ast\leftarrow S$, best objective value $Z^\ast\leftarrow f\left(S\right)$, where $f\left(S\right)$ is an evaluation function;
	\item Generate a set of states of size $r$ which is denoted as $\boldsymbol{S}^{\prime}=\left\{S_1,\ S_2,\ldots,S_r\right\}$, in which each state $S_i,\ i=1,2,\ldots,r$ is `close' to state S;
	\item Find $\hat{S}=\arg{\min_{S_i}{f\left(\boldsymbol{S}^\prime\middle| S_i\in\boldsymbol{S}^\prime,\ S_i\notin\mathcal{L}\right)}}$ and $\hat{Z}=f\left(\hat{S}\right)$;
	\item $S^\ast \leftarrow\hat{S}, Z^\ast \leftarrow\hat{Z}$ if $\hat{Z}<Z^\ast$;
	\item Update tabu list $\mathcal{L}$, $S\leftarrow\hat{S}$;
	\item Stop when $n>\phi$, otherwise $n \leftarrow n+1$ and go back to Step 2.
\end{enumerate}

Tabu search can select the obtained state in each iteration even though the candidate state is worse than the current best state. This strategy can help tabu search escape from a local optima so that there is a chance to search for a global optima. With this strategy, we introduce our proposed solution algorithm as follows. The topology of spanning tree $T$ can be considered as state $S$, and $Z(T)=f\left(T\right)$ means the value of objective function, $Z\left(T\right)$ is defined as the total passenger-kilometer in $T$. Thus, Step 1 in the general tabu search procedure can be updated by setting $\mathcal{L}=\emptyset, {_n}T={_0}T, T^\ast={_0}T$, and $Z^\ast=Z\left({_0}T\right)$, where $\emptyset$ is the empty set. ${_0}T$ can be a random spanning tree or a minimum spanning tree \citep{Oncan2008, Rothlauf2009}, which can be generated by an efficient spanning tree algorithm, e.g., Kruskal algorithm \citep{Kruskal1956}. Select r spanning trees $\boldsymbol{T}^\prime=\left\{T_1,T_2,T_3,\ldots,T_r\right\}$ generated by swapping a pair of links in T (as Step 2), where the swapping procedure involves removing an existing link then adding a new link. In Step 3, tabu search calculates the expected passenger-kilometers for each spanning tree in $\boldsymbol{T}^\prime$, and sets $\hat{T}=\arg{\min_{T_i}{f\left(\boldsymbol{T}^\prime\middle| T_i\in\boldsymbol{T}^\prime,\ T_i\notin\mathcal{L}\right)}}$. With these modifications, we can apply the tabu search to solve TND-STA for a promising solution.

For efficient computation of the TND-STA using the tabu search, we propose an algorithm by studying the properties of a tree. In Step 2 of the tabu search, a candidate spanning tree is generated by implementing the swapping procedure in ${_n}T$. To verify a generated graph is a tree, we need to check whether the graph contains a closed path, where a closed path is a directed path that starts at one node and ends at the same node.

In this paper, the swapping procedure is first to remove a link from a spanning tree to produce two subnetworks. Then, we add a new link to connect one node in a subnetwork and another node in the other subnetwork to produce a new network. The new network constructed through this swapping procedure is always a spanning tree which is shown in Proposition 1. To prove Proposition 1, we need the following lemmas.

\begin{lemma}
  A spanning tree $T\left(V,E\right)$ is always divided into two disjoint subnetworks when a link is deleted from it, where $V$ and $E$ are set of nodes and links in the spanning tree, respectively.
  \label{lem:tree_divided_two_subnetworks}
\end{lemma}
\begin{proof}
  A graph with one link deleted from the global tree $T\left(V,E\right)$ becomes disconnected \citep{Diestel2010}. Hence, the deletion of one link in a tree always generates two disjoint subnetworks.
\end{proof}

\begin{lemma}
  Any subgraph of a spanning tree is a tree.
  \label{lem:subgraph_of_tree_is_tree}
\end{lemma}
\begin{proof}
  Consider a spanning tree $T\left(V,E\right)$ and an arbitrary subgraph $\dot{G}\left(\dot{V},\dot{E}\right)$ with $\dot{E}=\left\{e_1,e_2,\ldots,e_k\right\},\bar{E}=E-\dot{E}=\left\{e_{k+1,}e_{k+2},\ldots,e_{\left|V\right|-1}\right\},\dot{V}=\left\{v_1,v_2,\ldots,v_q\right\},\bar{V}=V-\dot{V}={v_{q+1},v_{q+2},\ldots,v_{\left|V\right|}}$. From the definition of a tree, it is sufficient to show that $\dot{G}\left(\dot{V},\dot{E}\right)$ has no closed path and the number of links in $\dot{E}$ is $\left|\dot{V}\right|-1$. First, if $\dot{G}\left(\dot{V},\dot{E}\right)$ has a closed path, then the graph with all the additional links in $\bar{E}$, i.e., $T\left(V,E\right)$, always has a closed path that leads to a contradiction. Thus, the subgraph $\dot{G}\left(\dot{V},\dot{E}\right)$ of the spanning tree always has no closed path. Then, consider the operation of adding all links to $\dot{G}\left(\dot{V},\dot{E}\right)$ by repeating the addition of the links contained in $\bar{E}$. The graph generated by this operation is $T\left(V,E\right)$. Total number of this operation is $\left|\bar{E}\right|=\left|V\right|-1-k$, and the number of nodes added to the graph must be $\left|\bar{V}\right|=\left|V\right|-q$. Both of these conditions are satisfied only when $k=q-1$. Therefore, $\left|\dot{E}\right|=\left|\dot{V}\right|-1$ is valid.
\end{proof}

\begin{lemma}
  ${^1}T_{-a}$  and ${^2}T_{-a}$  are disjoint subnetworks when link $a$ is removed from spanning tree $T\left(V,E\right)$.
  \label{lem:subnetworks_disjoint}
\end{lemma}
\begin{proof}
  A direct result from Lemma~\ref{lem:tree_divided_two_subnetworks} and Lemma~\ref{lem:subgraph_of_tree_is_tree}.
\end{proof}

\begin{lemma}
  The graph obtained by connecting one node in the spanning tree ${^1}T_{-a}$  and another node in the spanning tree ${^2}T_{-a}$  is a spanning tree, where ${^1}T_{-a}$  and ${^2}T_{-a}$  are trees by removing link $a$ in the spanning tree $T\left(V,E\right)$.
  \label{lem:connected_graph_is_tree}
\end{lemma}
\begin{proof}
  It is sufficient to show that the graph obtained by connecting one node in ${^1}T_{-a}$  and another node in ${^2}T_{-a}$ has no closed paths and the number of links is $|V|-1$. By Lemma~\ref{lem:subnetworks_disjoint}, ${^1}T_{-a}$ and ${^2}T_{-a}$ are trees by removing link $a$ in the spanning tree $T\left(V,E\right)$. Therefore, the total number of links in these two subnetworks is $\left|V\right|-2$, and the number of links in the graph with one additional link is $\left|V\right|-1$. By Lemma~\ref{lem:subgraph_of_tree_is_tree}, each of the two trees does not contain a closed path. Thus, the graph by connecting one node in ${^1}T_{-a}$ and another node in ${^2}T_{-a}$  (i.e., adding a new link between the two disjoint trees) has no closed path.
\end{proof}

\begin{proposition}
  A spanning tree is always partitioned into two subnetworks when a link is removed from it, and the graph obtained by adding a new link that connects a pair of nodes with one in each subnetwork is a spanning tree.
  \label{prop:swapping_procedure}
\end{proposition}
\begin{proof}
  A direct result from Lemma~\ref{lem:tree_divided_two_subnetworks} and Lemma~\ref{lem:connected_graph_is_tree}.
\end{proof}

From Proposition~\ref{prop:swapping_procedure}, we first remove a link in the tree, then reconnect these two subnetworks by linking one node in a subnetwork and one node in the other subnetwork. The number of candidate links that reconnect two subnetworks after removing link $a$ is:

\begin{equation}
K_a=m\left({^1}T_{-a}\right)\times m\left({^2}T_{-a}\right),\label{eq:number_of_links_to_reconnect_subnetworks}
\end{equation}

where $m\left(T\right)$ is the number of nodes in tree $T$. The maximum value of $K_a$ is $\left\lfloor\frac{m\left(T\right)}{2}\right\rfloor\left\lceil\frac{m\left(T\right)}{2}\right\rceil$ by Lemma~\ref{lem:max_number_of_links}, where $\left\lfloor x\right\rfloor$ is the largest integer no greater than $x$ and $\left\lceil x\right\rceil$ is the least integer no smaller than $x$.

\begin{lemma}
  Given a positive integer $n$, $C_1^{n-k}C_1^k$ achieves its maximum provided that
  $k=\left\lfloor\frac{n}{2}\right\rfloor\ \text{or}\ k=\left\lceil\frac{n}{2}\right\rceil$.
  \label{lem:max_number_of_links}
\end{lemma}
\begin{proof}
  Since $C_r^n=\frac{n!}{r!\left(n-r\right)!}$, then $C_1^{n-k}C_1^k=-k^2+nk$. Denote $f\left(k\right)=-k^2+nk$. Since $f\left(k\right)$ is a concave function, relaxing $k$ as a real number and solving $\frac{df}{dk}=0$, we have $k=\frac{n}{2}$.
\end{proof}

Denote link $b$ as the link that reconnects the two disjoint subnetworks after removing link $a$ in a spanning tree. Then, number of all possible swapping link pairs $\left(a,b\right)$ in spanning tree ${_n}T$ is:
\begin{equation}
l_{{_n}T}=\sum_{a\in E\left({_n}T\right)}{m\left({_n^1}T_{-a} \right)\times m\left({_n^2}T_{-a} \right)}\label{eq:number_of_swapping_link_pairs}
\end{equation}

where $E\left({_n}T\right)$ is the set of links in ${_n}T$.

\subsection{Calculation of the objective function after implementing the swapping procedure}

The objective function, i.e., Formulation~\ref{eq:tnd_sta_obj}, can be rewritten as follows:
\begin{equation}\label{form:objective_function}
  Z\left(a,b|{_n}T\right)=\boldsymbol{c}\left({_n}T_{-a}^{+b}\right)^T\boldsymbol{d}.
\end{equation}

This formulation requires the shortest path calculation to obtain all distances of node pairs in the spanning tree ${{_n}T}_{-a}^{+b}$ that first removes link $a$ then inserts link $b$. 
\figref{fig:removing_link_and_reconnect_candidates} shows the removal of link $a$ and all candidates for link $b$ that can reconnect the two disjoint subnetworks. Indeed, we do not need to calculate the shortest paths for all node pairs. The idea is simple, as the shortest path from one subnetwork to the other must traverse the new link, i.e., link $b$. This implies that the shortest path calculation must have been carried out only once within each subnetwork. To show this, we rewrite our objective function as follows:

\begin{figure}[htbp]
  \centering
  \begin{minipage}[b]{0.45\linewidth}
    \centering
    \includegraphics[keepaspectratio, scale=0.4]{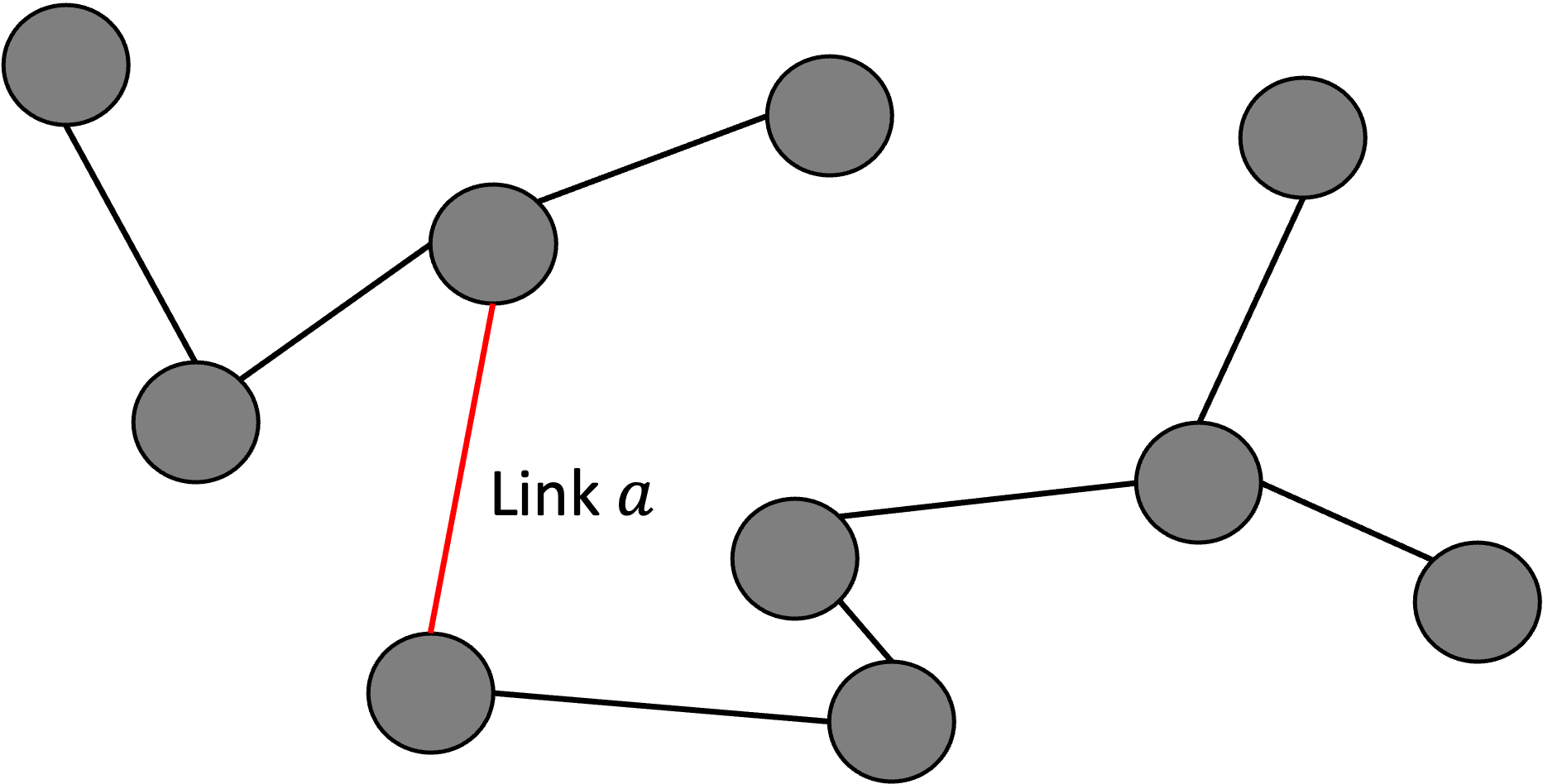}%
    \caption*{(a)	Removing a link from a spanning tree}%
  \end{minipage}%
    \hfill
  \begin{minipage}[b]{0.45\linewidth}
    \centering
    \includegraphics[keepaspectratio, scale=0.4]{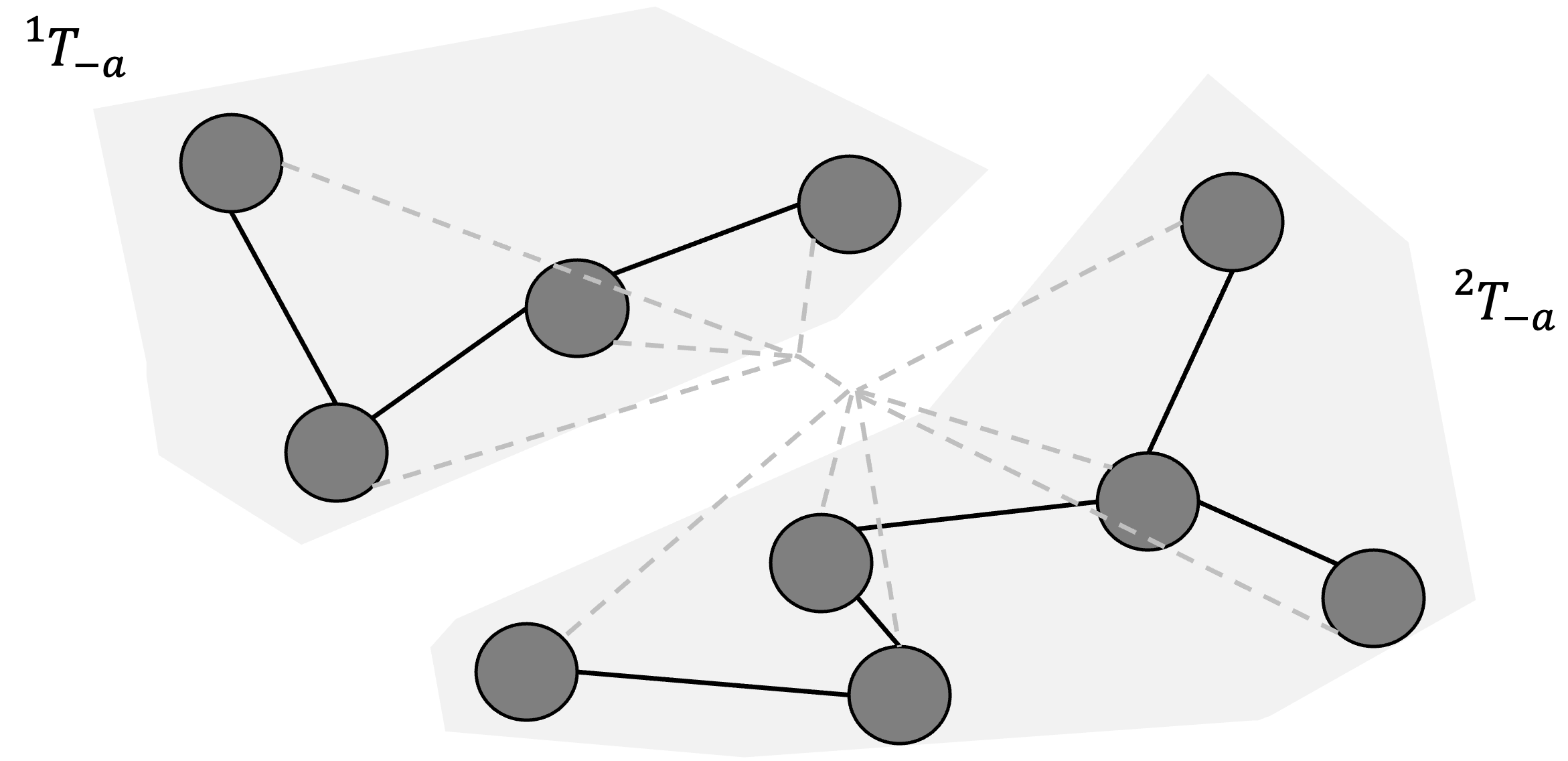}%
    \caption*{(b)	All candidate links that can reconnect two disjoint subnetworks}%
  \end{minipage}
  \caption{Example of removing a link in a spanning tree and listing all candidate links that can reconnect the two disjoint subnetworks.}\label{fig:removing_link_and_reconnect_candidates}%
\end{figure}

\begin{align}\label{eq:objective_function_rewritten}
Z\left(a,b\mid{_n}T\right) =\sum_{i,j\in N\left({_n^1}T_{-a}\right),i\neq j}{c_{\left(ij\right)}\left(d_{\left(ij\right)}+d_{\left(ji\right)}\right)} +\sum_{i,j\in N\left({_n^2}T_{-a}\right),i\neq j}{c_{\left(ij\right)}\left(d_{\left(ij\right)}+d_{\left(ji\right)}\right)}
\\
+\sum_{i\in N\left({_n^1}T_{-a}\right),j\in N\left({_n^2}T_{-a}\right)}{\left[\left(c_{\left(ib\left(1\right)\right)}+c_{b\left(2\right)j}\right)\left(d_{\left(ij\right)}+d_{\left(ji\right)}\right) +t_{\left(b(1)b(2)\right)}\left(d_{(ij)}+d_{(ji)}\right)\right]}, \notag
\end{align}

where $N\left(T\right)$ denotes the all the nodes in tree $T$. (Note: For better legibility, we use $c_{\left(ij\right)}$ and $d_{\left(ij\right)}$, i.e., their matrix notation, instead of $c_w$ and $d_w$, i.e., their vector notation in the TND-STA model.) Without loss of generality, we assume $b\left(1\right)\in N\left({_n^1}T_{-a} \right)$ and $b\left(2\right)\in N\left({_n^2}T_{-a} \right)$. The first and second term represent the shortest path lengths between node pairs within Subnetworks 1 and 2, respectively. For a node pair $(ij)$ with $i$ in Subnetwork 1 and $j$ in Subnetwork 2, the length of the shortest path is the sum of the length of the shortest path from node $i$ in Subnetwork 1 to $b\left(1\right)$, length of link $b$, and the length of the shortest path from node $j$ in Subnetwork 2 to $b\left(2\right)$, which is the third term in the objective function. Although it is assumed here that $c_{\left(ij\right)}=c_{\left(ji\right)}$, this assumption can be easily relaxed if it is not true.

\subsection{Updating tabu list}

In the swap operation, there might be a case that we would repeatedly swap a link multiple times.
To avoid this, tabuing approach is applied. Tabu is defined by a swapping link pair: a removal link and an insertion link. The length of the tabu list can be set as any positive integer. 
The tabu list is sorted in chronological sequence. If the tabu list is full, the oldest swapping link pair is removed from the tabu list when adding a new pair. 
This procedure is expected to lead to an optimal solution, even when the shape of the tree changes during the search process or when swapping is recorded in the tabu list based on previous searches, as governed by this rule. Based on this consideration, the swap is executed even if the swap is in the tabu list when the swap can update the optimal solution. 
Tabu search has been applied to the traveling salesman problem in \citet{Tsubakitani1998}. 
Their results show that the performance of the computation time deteriorates significantly when the size of the tabu list is larger than $N/2$. 
Therefore, for simplicity, we use $N/4$ in the proposed method.

\subsection{Pseudocode of Link Swapping with Tabu Search}
The pseudocode of the algorithm of TND-STA is shown in \figref{fig:pseudo_code}

\begin{figure}[!h]
\begin{algorithm}[H]
  \caption{Link Swapping with Tabu Search}
\begin{algorithmic}[1]
  \footnotesize{
    \Require{$\boldsymbol{t}$ -- Vector of link travel distances, $\boldsymbol{d}$ -- Vector of OD demand, $\boldsymbol{\phi}$ -- maximum number of iterations, $\boldsymbol{\psi}$ -- number of candidate links for deletion in the original spanning tree}
    \Ensure{$T^{*}$ -- Optimal spanning tree}}
    \State{$n$ := $0$} 
    \State{$_nT$ :=  \Call{Kruskal}{$\boldsymbol{t}$}}
    \State{$T^{*}$ := $_nT$}
    \State{$\boldsymbol{c}$:= \Call{Dijkstra}{$T^{*}$}}
    \State{$Z$:= $Z(_nT)$}
    \State{$\mathcal{L}:=\emptyset$}
    \While{$n<\phi$}
    \State{Randomly select $\psi$ links in $_nT$ and store them in $\mathcal{A}$}
    \State{$\boldsymbol{\gamma} \gets \emptyset$}
      \ForAll{$a \in \mathcal{A}$}
        \State{$\{{_n^1}T_{-a},{_n^1}T_{-a}\}\gets _nT-a$}
        \State{Identify all candidate links that can reconnect ${_n^1}T_{-a}$ and ${_n^2}T_{-a}$ and store them in $\mathcal{B}_a$}
        \ForAll{$b \in \mathcal{B_a}$}
          \State{$_nT_{-a}^{+b}\gets \{{_n^1}T_{-a},{_n^1}T_{-a}\}\cup b$}
          \State{$\gamma_{(ab)} \gets Z(_nT_{-a}^{+b})$}
          \State{$\boldsymbol{\gamma}\gets \boldsymbol{\gamma}\cup \gamma_{(ab)}$}
        \EndFor
      \EndFor
      \State{$(a',b')=\argmin_{a\in \mathcal{A},b\in \mathcal{B}_a}{\left\{\gamma_{(ab)}\mid \gamma_{(ab)}\in \boldsymbol{\gamma}\right\}}$}
      \State{$\gamma_{(a'b')}=Z(_nT_{-a'}^{+b'})$}
      \If{$\gamma_{(a'b')}<Z^{*}$}
        \State{$Z^{*}\gets \gamma_{(a'b')}$}
        \State{$T^{*}\gets _nT_{-a'}^{+b'}$}
        \State{$_{n+1}T \gets _nT_{-a'}^{+b'}$}
      \Else 
        \ForAll{$\gamma \in \boldsymbol{\gamma}$}
          \If{$(a',b') \notin \mathcal{L}$}
            \State{$_{n+1}T \gets _nT_{-a'}^{+b'}$}
            \State{\textbf{exitfor}}
          \Else
            \State{$\boldsymbol{\gamma} \gets \boldsymbol{\gamma}-\gamma_{(a'b')}$}
          \EndIf
        \EndFor
        \State{$\mathcal{L}\gets \mathcal{L}\cup \{(a',b')\}$ (Note: If $\mathcal{L}$ is full, remove the oldest link pair, then update $\mathcal{L}$ by including $(a',b')$.)}
      \EndIf
      \State{$n\gets n+1$}
    \EndWhile
\end{algorithmic}
\end{algorithm}
\caption{The pseudo code of the proposed method}
\label{fig:pseudo_code}
\end{figure}

The full description of Link Swapping with Tabu Search is as follows. 
Lines 1 -- 6 declares the initial state of the proposed algorithm. 
The proposed algorithm adopts Krulskal's algorithm to output a minimum spanning tree and set the tree as the global optimum spanning tree \citep{Oncan2008, Rothlauf2009}.
Then, it adopts an efficient shortest path algorithm, e.g., Dijkstra's algorithm \citep{Knuth1977, Wang2020}, to calculate the shortest path distance between any pair of nodes in the spanning tree. 
After that, the proposed algorithm evaluates the quality of the current spanning tree by using Formula (1) in the TND-STA model to obtain the global optimum objective value. In Lines 8 -- 20, the proposed algorithm randomly chooses $\psi$ exiting links in the current spanning tree and store them in set $\mathcal{A}$. 
For each existing link $a$ in $\mathcal{A}$, the proposed algorithm evaluates each candidate link b that can reconnect the two disjoint networks using Formula (12). 
After evaluating each chosen candidate swapping link pair $(a,b)$, the proposed algorithm searches for the current optimum spanning tree, i.e., Lines 19 -- 20. Then, the proposed algorithm proceeds to check whether the current optimal solution improves the global optimum solution. 
In Lines 21 -- 24, if the objective value of the current optimum spanning tree is smaller than the global optimum objective value, the proposed algorithm updates the global optimum objective value and the global optimum spanning tree using the current optimum spanning tree. 
Also, the proposed algorithm uses the current optimum spanning tree as the starting point for its next iteration and updates the tabu list by including the associated swapping link pair $(a',b')$. 
However, if there is no improvement at the current iteration, the proposed algorithm proceeds to Lines 26 -- 34 examines the tabu list using the strategy mentioned in Section 2.4. 
The proposed algorithm repeats procedures in Lines 7 -37 until the While-loop counter $n$ exceeds the maximum number of iterations $\phi$.

\section{Simulation using Canberra bus network data}
The efficacy of our proposed Link Swapping with Tabu Search is verified using the case study of Canberra bus network.

\subsection{Canberra bus network data}
Canberra, the Australian capital city, has a population of approximately 460,900 \citep{ABS2022} which is spread across an area of 807.6 km2. Canberra has somewhat a unique structure -- a Y-shaped poly-centric city is characterized by small town centers and residential districts based on the garden city concept. The main public transportation system in Canberra is the Australian Capital Territory (ACT) Internal Omnibus Network (ACTION). It provides bus transfer and transit services to Canberra suburbs including a regional community minibus service. There is a regional train service but not used for intra-city transit. The ACTION offered four rapid bus routes, a free city loop and 89 routes on weekdays, and two rapid bus routes and 47 routes during the weekend (as of November 2017) . Canberra is a sparse city with a low population density, and the districts are clearly divided because the city growth is strictly controlled by the pre-defined city plan. We therefore focus on the connection of each suburb in the bus network rather than a route with detailed road sections. We study the bus service improvement by comparing the optimal solution obtained from the proposed Link Swapping with Tabu Search and the operated bus route. The passenger origin-destination (OD) demand is obtained from card usage data.

The smart card data used for this research was recorded in June 2016 by the Australian Capital Territory Government. The smart card data recorded 1207494 trips on public transport, collected by an automatic fare collection system which is used on every ACTION BUS vehicle. Passengers tap on when they board and tap off when they alight a bus. Each record contains the attributes: bus route number, origin date, origin tap on time, origin stop name, origin stop coordinate, destination tap off time, destination stop name, destination stop coordinate, and type of smart card.

Bus passengers normally board at a bus stop on a street and alight at another bus stop that is located at the opposite side of the same street. This feature leads to a highly asymmetric OD matrix as passengers board and alight at different bus stops which are just located on the opposite side of the same street. Since our objective is to analyze the connection between the suburbs of the bus network, we do not need the detailed station information within a suburb. Thus, number of nodes in the bus network is reduced to 111 large bus stops, as there are 111 suburbs in Canberra. There are two additional advantages in this aggregation. One is that the OD matrix becomes almost symmetric when considering 111 bus stops. The other is that this aggregation can reduce the computation complexity so that we can compare the performance of the Link Swapping heuristic algorithm (Heuristic 1) and Link Deletion heuristic algorithm (Heuristic 2) developed by \citet{Bell2020} with our proposed solution algorithm: Link Swapping with Tabu Search.

In order to demonstrate the superiority of the proposed method, we compare the proposed heuristic algorithm and the \citet{Bell2020} heuristic algorithms in terms of computation time and objective value. In Heuristic 1, the number of iterations is 6105 (= 111 x 110 / 2) because the procedure is iterated until links between all node pairs are tabooed. In Heuristic 2, the spanning tree is constructed by link deletion one by one from the complete graph of 111 bus stops until it becomes a tree, therefore the number of iterations is 5995 (= 6150 -- 110). In the tabu search, the number of iterations is a hyperparameter, which is set to 3000 such that the update status of the solution becomes sufficiently stable. The number of iterations of the extended Heuristic1 is also set as 3000, same as Tabu search. Other hyperparameters in the tabu search are set by trials. Number of random removal links $\psi$ is set as 7 and the size of tabu list |L| is set as 80. Since the proposed tabu search algorithm adopts a random strategy in Line 7, the computation time and the objective value are obtained by averaging 100 simulations using the algorithm. The numerical experiments are conducted in the environment -- CPU: AMD Ryzen3800x, RAM: DDR4-3200 64GB, OS: Windows 10, Coded on Matlab 2019b.

\tabref{tab:computational_result} summarizes the computation time and the objective value using Link Swapping heuristic algorithm (Heuristic 1), Link Deletion heuristic algorithm (Heuristic 2), and Link Swapping with Tabu Search. The computation time of our proposed tabu search algorithm is dramatically smaller than those using the other two algorithms, which is 0.22\% compared to Heuristic 1 and 0.30\% compared to Heuristic 2. The objective value obtained from the proposed algorithm is also smaller than those obtained from \citet{Bell2020} algorithms. Therefore, the proposed algorithm leads to a better solution in terms of computation time and objective value.

\begin{table}[h]
\caption{Computation time and objective value using Link Swapping heuristic algorithm, Link Deletion heuristic algorithm, and Link Swapping with Tabu Search.}
\label{tab:computational_result}
\centering
\begin{tabular}{l p{8em} p{7em} p{8em} } \hline
   & Heuristic 1 (Link Swapping) & Heuristic 2 (Link Deletion) & Link Swapping with Tabu Search \\ \hline
  Computation time (in seconds) & 32,237.86 &	23,047.85	& 69.93 \\
  Objective value (in $10^3$ km) & 547,418 &	576,220	& 544,633 \\ \hline
\end{tabular}
\end{table}

\figref{fig:convergence_result} shows the relationship between number of iterations and objective value of each algorithm. The proposed tabu search algorithm reaches the minimum objective value after a few iterations, while the Link Swapping heuristic algorithm (Heuristic 1) reaches its minimum objective value after about 1500 iterations. When reaching a spanning tree using the Link Deletion heuristic algorithm (Heuristic 2), its objective value increases abruptly, as the shortest path becomes much longer after removing many links in the network.

\begin{figure}[htp]
  \centering
  \includegraphics[keepaspectratio, scale=0.6]{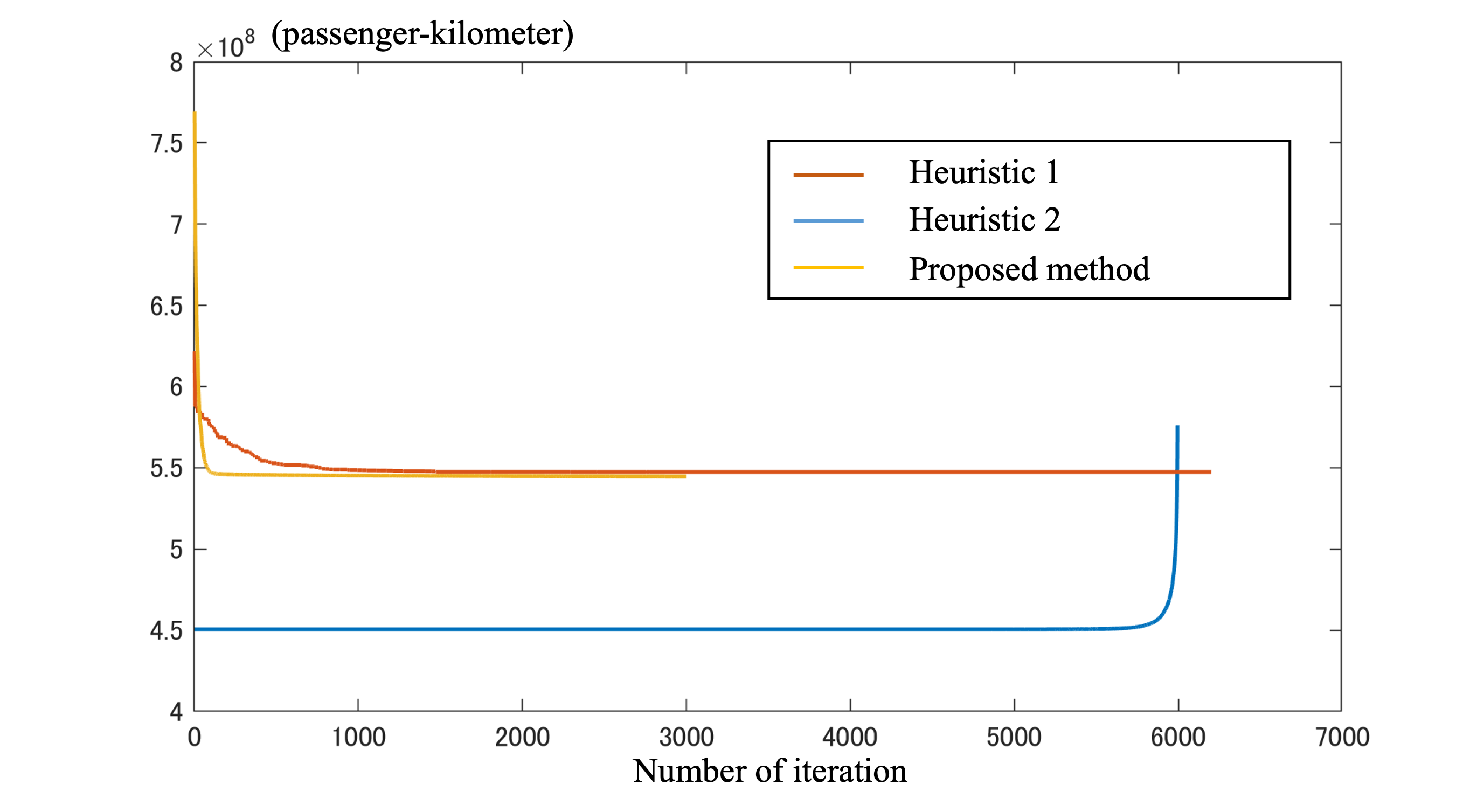}
  \caption{Convergence analysis of Link Swapping heuristic algorithm (Heuristic 1), Link Deletion heuristic algorithm (Heuristic 2), and Link Swapping with Tabu Search (Proposed method).}
  \label{fig:convergence_result}
\end{figure}

\figref{fig:computation_time_result} shows the relationship between computation time and objective value of each algorithm. Note that x-axis in \figref{fig:computation_time_result} is shown in $\log{10}$ scale. The proposed tabu search algorithm converges to a solution quickly. However, \citet{Bell2020} heuristic algorithms require more time to converge to a solution, as computing the inverse of a 111 x 111 matrix is time-consuming.

\begin{figure}[htp]
  \centering
  \includegraphics[keepaspectratio, scale=0.6]{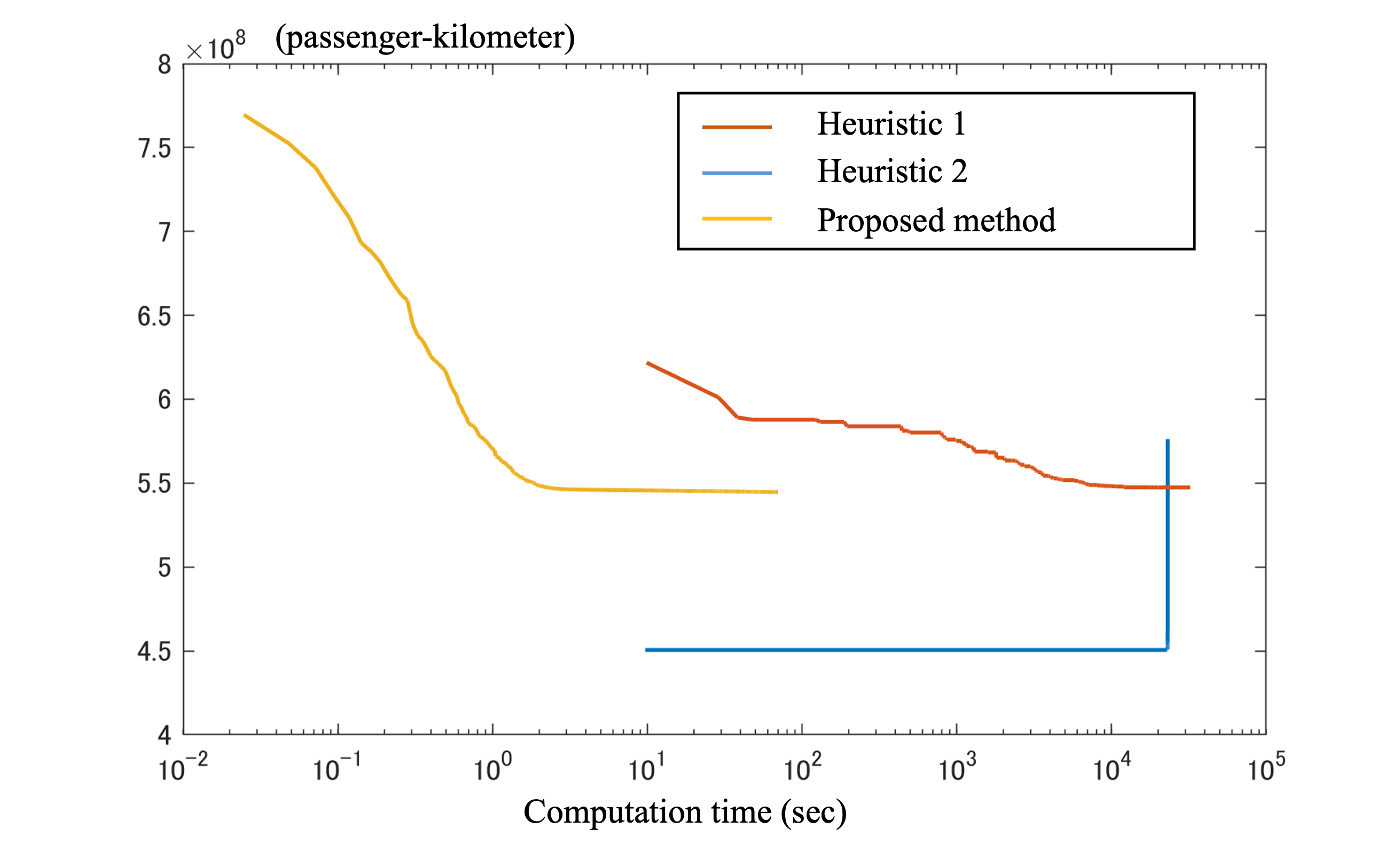}
  \caption{Performance analysis of Link Swapping heuristic algorithm (Heuristic 1), Link Deletion heuristic algorithm (Heuristic 2), and Link Swapping with Tabu Search (Proposed method).}
  \label{fig:computation_time_result}
\end{figure}

Therefore, the proposed tabu search algorithm -- Link Swapping with Tabu Search -- demonstrates its superiority in terms of computation time and solution quality.

\subsection{Comparison of different spanning trees}
This section compares the minimum passenger-kilometer spanning tree (MPKST) obtained using the proposed tabu search algorithm, the minimum distance spanning tree (MST), and the maximum demand spanning tree (MDST) in terms of total passenger-kilometer and topology. MST can be obtained by solving the following optimization problem:

\begin{equation}
  \min_{\boldsymbol{y}} \sum_{i,j \in N, i\neq j} t_{(ij)} y_{(ij)}
  \label{eq:minimum_distance_spanning_tree}
\end{equation}
subject to Formulae \eqref{eq:tnd_sta_num_links}, \eqref{eq:tnd_sta_no_cycles}, and \eqref{eq:tnd_sta_binary_y}.

MDST can be obtained by solving the following optimization problem:
\begin{equation}
  \max_{\boldsymbol{y}} \sum_{i,j \in N, i\neq j} -d_{(ij)} y_{(ij)}
  \label{eq:maximum_demand_spanning_tree}
\end{equation}

subject to Formulae \eqref{eq:tnd_sta_num_links}, \eqref{eq:tnd_sta_no_cycles}, and \eqref{eq:tnd_sta_binary_y}. (Note: To better reveal station pairs, we use $d_{\left(ij\right)}$, i.e., its matrix notation, instead of $d_w$, i.e., its vector notation in the TND-STA model.)

These two problems -- MST and MDST -- can be solved using Kruskal's algorithm \citep{Kruskal1956}. Table 4 summarizes the total passenger-kilometers obtained from the three spanning trees.  Compared with MST, MDST reduces passenger-kilometers by 33\%. 
Therefore, considering OD demand alone in designing a bus network is better than just considering minimum travel distance. 
However, considering both OD demand and cost of travelling is paramount in a bus network, as MPKST abruptly reduces passenger-kilometers by 77\% compared with MST and by 66\% compared with MDST.

\begin{table}[h]
\centering
\caption{Total passenger-kilometers obtained from the three spanning trees: MST, MDST and MPKST.}
\label{tab:res_spanning_trees}
\begin{tabular}{p{0.1\textwidth} p{0.15\textwidth} p{0.25\textwidth} p{0.25\textwidth}}
\hline
Spanning tree & Total passenger-kilometers & Percentage change with respect to MST & Percentage change with respect to MDST \\
\hline
MST & 2,375,188,779 & - & +49.31\% \\
MDST & 1,590,809,637 & -33.02\% & - \\
MPKST & 544,329,352 & -77.08\% & -65.78\% \\
\hline
\end{tabular}
\end{table}

\figref{fig:three_spanning_trees} illustrates the topologies of the three spanning trees. The MPKST (\figref{MPKST}) is composed of a trunk line that connects high demand nodes and several branch lines. The color of a node in the figure indicates the degree of the demand. MPKST indicates that the red nodes with higher degrees are shown at the major traffic points such as Belconnen, City and Woden Valley. Consequently, the location of high-degree nodes identified by MPKST is actually an important traffic node, which is consistent with the purpose of this study to look for traffic nodes. And the trunk line formed with these three major traffic points is closely in line with the current rapid service bus route in \figref{fig:canberra_rapid_network}. This finding shows that our proposed methodology demonstrates the practicability in designing a bus network.

The maximum demand spanning tree (MDST) captures large demand nodes directly (\figref{MDST}). Consequently, it also reduces the passenger-kilometer, as residents in Belconnen can enjoy the direct service to City. In addition to the Belconnen-City link, the MDST also identifies another two large demand links -- one is between City and Woden Valley, and the other is between City and Tuggeranong. Number of residents in Woden Valley is relatively small (16000 in 2016), but the proportion of bus usage for commute is slightly higher than Belconnen (also at about eight percent). Tuggeranong has the second largest resident group in the area (around 42000) but smaller proportion of bus usage for commute (at 5.2 percent). However, the only major district with high bus users that is not being identified by the MDST is Gungahlin, which is captured in the trunk line in MST.

Another issue with MDST shown in \figref{MDST} is the separation of the line to Woden and Tuggeranong. Although both Woden Valley and Tuggeranong show large demand to City, the line to get to Tuggeranong is very close to that to Woden Valley which is captured in MPKST in \figref{MPKST}. Therefore, MPKST inherits the advantages in both the maximum demand (captured by MDST) and the minimum distance spanning trees (captured by MST) to provide the minimum travel distance in the largest demand areas. The proposed method, i.e., the MPKST, can identify the smallest aggregated passenger-kilometer bus line in Canberra. This major bus line that serves as the public transport corridor is identified inthe path Belconnen-City-Woden Valley-Tuggeranong (combining the major links: Belconnen-City and City-Woden Valley inthe trunk line, andthe major link City-Tuggeranong in a branch line). Such a bus line isthe blue rapid line shown in \figref{fig:canberra_rapid_network}. This line itself resemblesthe western andthe southern parts ofthe public transport corridor inthe Y-shaped poly-centric city plan drawn forthe National Capital Development Commission (NCDC) inthe late 1960s and the development of Griffin original plan \citep{Wensing2013}.

Although MPKST also identifies the northern corridor to Gungahlin which is captured in MDST, it has not been classified as the major bus line. This may also explain why although gaining a considerable success, this rapid line is not as successful as the blue rapid line (the two rapids line can be seen in the 2016 network map). One of the possible explanations is that the Flemington road used as the main route of the northern red rapid lines was very congested. This reduced the efficiency of the bus system, as buses have to compete with cars for access to the road and even some commuters choose to deter from this road and take a further ring road outside Canberra to the Woden Valley area. This may have changed with the introduction of the light rail along the northern corridor from Gungahlin to City and increase the usage of public transport from Gungahlin in the northern area in Canberra.

\begin{figure}[p]
\begin{tabular}{cc}
      \begin{minipage}[t]{0.45\columnwidth}
        \centering
        \includegraphics[keepaspectratio, scale=0.4]{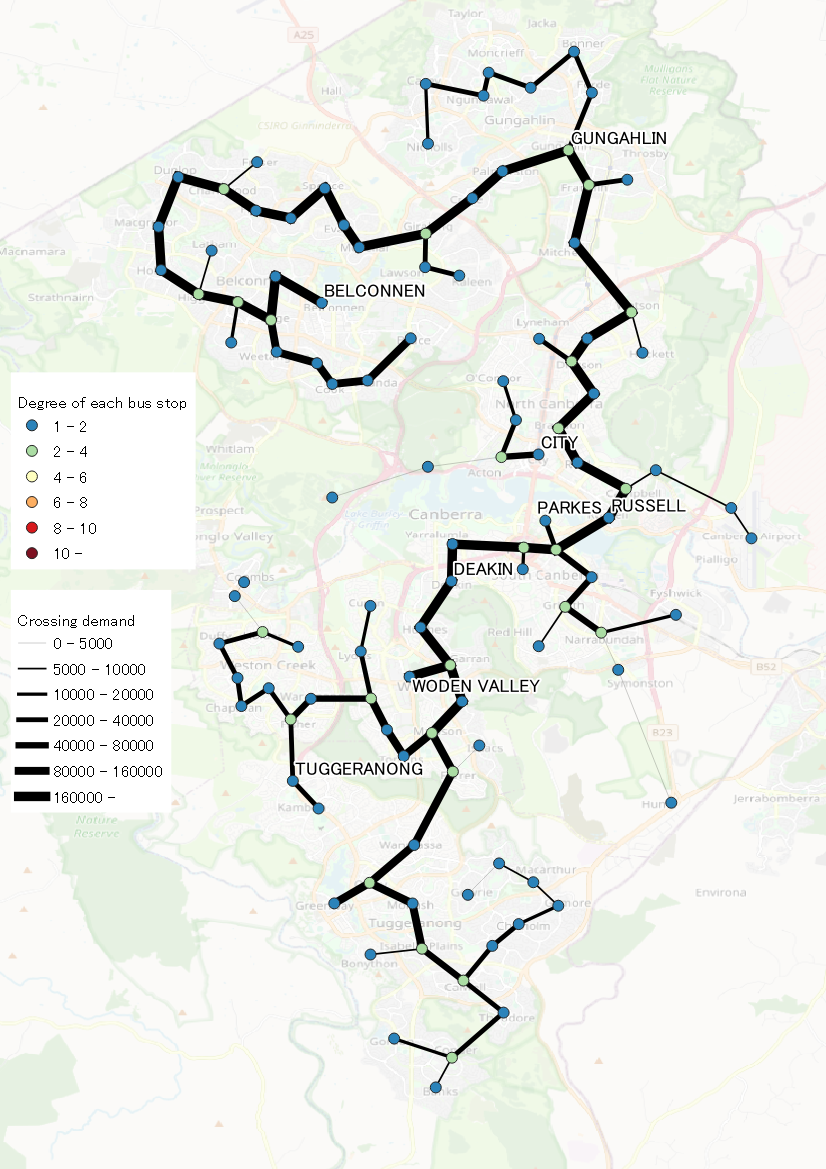}
        \subcaption{Minimum distance spanning tree (MST)}
        \label{MST}
      \end{minipage} &
      \begin{minipage}[t]{0.45\columnwidth}
        \centering
        \includegraphics[keepaspectratio, scale=0.4]{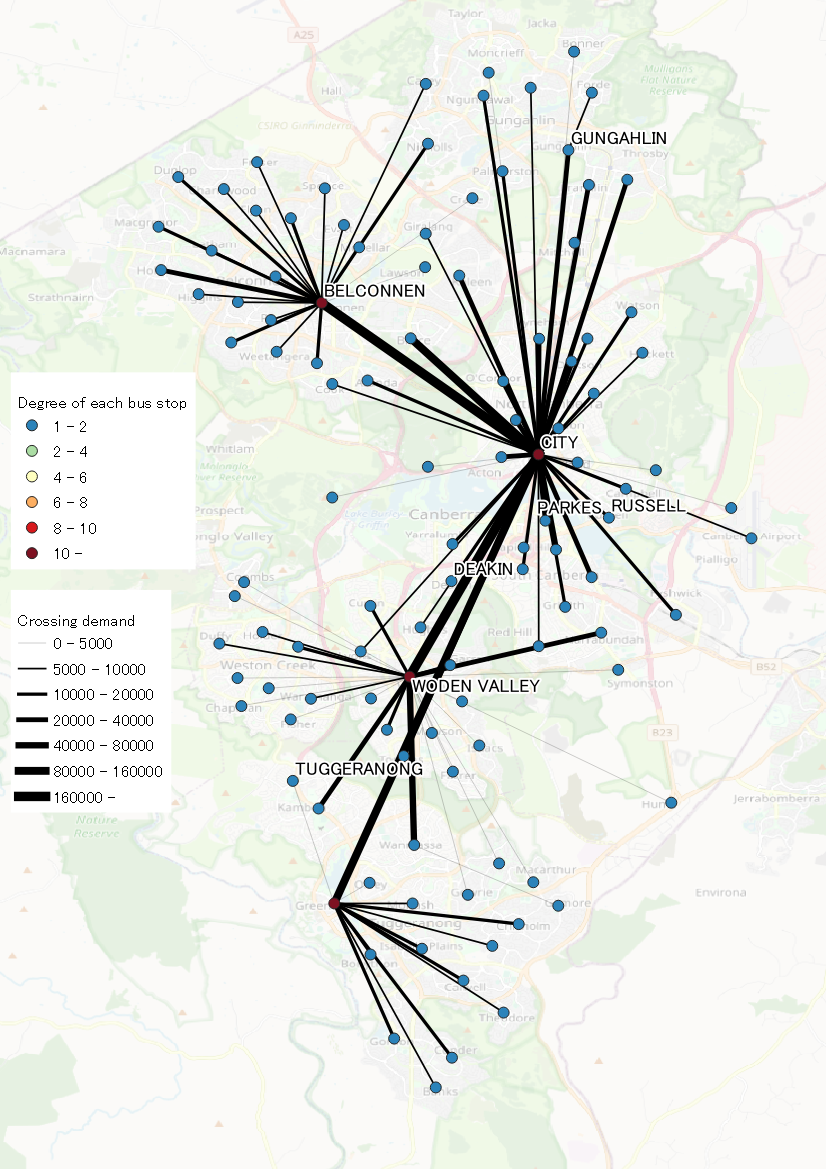}
        \subcaption{Maximum demand spanning tree (MDST)}
        \label{MDST}
      \end{minipage} \\
      \multicolumn{2}{c}{
      \begin{minipage}[t]{0.9\columnwidth}
        \centering
        \includegraphics[keepaspectratio, scale=0.4]{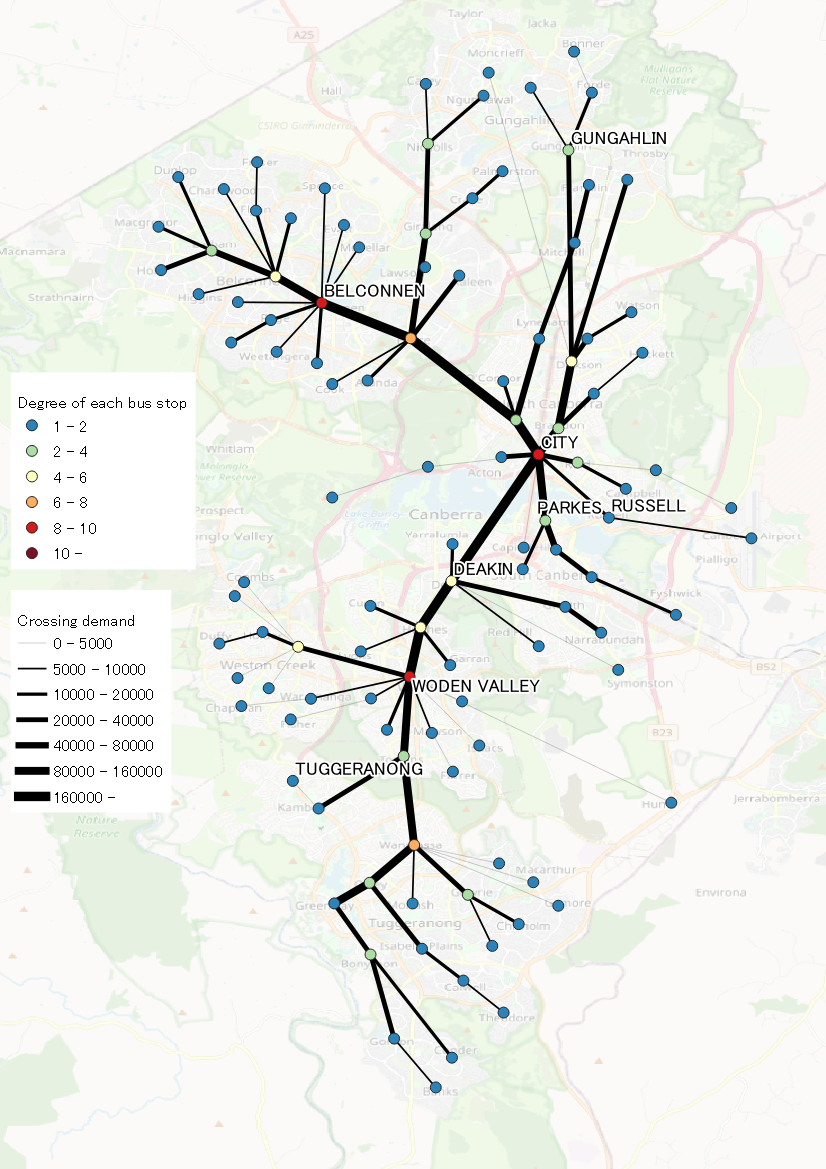}
        \subcaption{Minimum passenger-kilometer spanning tree (MPKST)}
        \label{MPKST}
      \end{minipage}}
    \end{tabular}
     \caption{Topologies of the three spanning trees: MST, MDST and MPKST.}\label{fig:three_spanning_trees}
\end{figure}

\begin{figure}[p]
  \centering
  \includegraphics[keepaspectratio, scale=0.8]{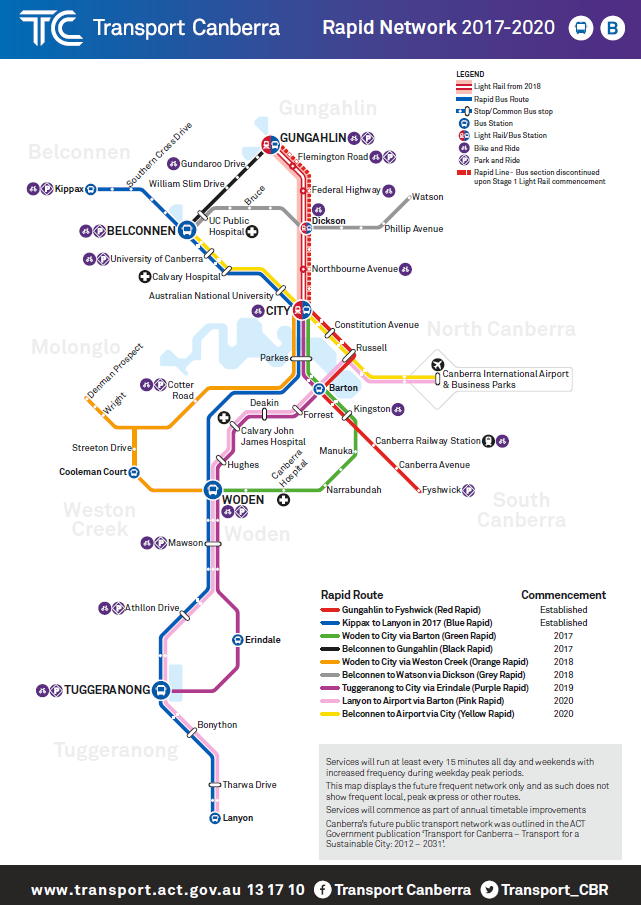}
  \caption{Map of Canberra rapid public transport network.}
  \label{fig:canberra_rapid_network}
\end{figure}

Moreover, MPKST captures other rapid lines and hubs in the rapid network shown in \figref{fig:canberra_rapid_network} and achieves the least passenger-kilometers among the three spanning trees. Also, MPKST confirms the practicality of the proposed methodology. Take a look at \figref{MPKST}, MPKST connects Deakin and Narrabundah, while, on the bus route map in \figref{fig:canberra_rapid_network}, Narrabundah is connected to Woden Valley. While the Green Rapid in \figref{fig:canberra_rapid_network} forms a closed path in combination with other lines, the MPKST is not allowed to make a closed path. Therefore, instead of Woden Valley and City, where Green Rapid connects with other lines, MPKST connects Deakin and Narrabundah, which are located in the middle. Similarly, the western line captures a large demand between Woden Valley and the western area that constitutes a closed path via the northern area. However, this cannot be captured in MPKST, as it is a tree that excludes closed paths.

Although the passenger-kilometer is minimized in MPKST, there is a concern that the path travel distance for passengers between certain origin-destination (OD) pairs may be large. If many passengers travel a larger distance on the MPKST compared to the direct distance, the level of service may be lower regardless of the frequency setting. Or if a particular OD pair travels an extremely long distance on the MPKST, it may be a concern in terms of fairness. We show how the ratio $\frac{c_{\left(ij\right)}}{t_{\left(ij\right)}}$  of the path travel distance $c_{\left(ij\right)}$ to the link distance $t_{\left(ij\right)}$ between OD pair on MPKST is distributed. (Note: For better legibility, we use $c_{\left(ij\right)}$, i.e., its matrix notation, instead of $c_w$, i.e., their vector notation in the TND-STA model.) \figref{fig:relationship_between_cumulative_frequency_and_ratio} shows the cumulative frequency distribution of the quantity of demand and OD pairs corresponding to the values of $\frac{c_{\left(ij\right)}}{t_{\left(ij\right)}}$. In MPKST, more than 80\% of the total demand is contained in $\frac{c_{\left(ij\right)}}{t_{\left(ij\right)}}\leq 1.5$, and more than 90\% in $\frac{c_{\left(ij\right)}}{t_{\left(ij\right)}}\leq 2.0$. Compared with the results of MDST and MST, many demands can reach the destination with a path travel distance close to the link distance for small values. Even though MDST is a structure that directly connects between ODs with large demand, MPKST has a larger cumulative frequency for $\frac{c_{\left(ij\right)}}{t_{\left(ij\right)}}\geq 1.3$. In \figref{fig:number_of_od_pairs}, the number of OD pairs shows that MPKST has better results than MST and MDST. \figref{fig:relationship_between_cumulative_frequency_and_ratio} also shows the results of a similar calculation on a network that tracing the actual public transport network in Canberra, shown in \figref{fig:actual_network}.

To ensure that all suburbs are connected, this network adds other bus routes to the rapid service routes in \figref{fig:canberra_rapid_network}. The number of links in \figref{fig:actual_network} is 149, which is 39 more than MPKST, and the distance is 18,355, which is 10\% more than MPKST's 16,732, but as can be seen in \figref{fig:cumulative_frequency_demand}, there is no significant difference in the distribution, especially in demand perspective. This result indicates that if the bus routes are routed on the links used in MPKST, the distance traveled by passengers will not change significantly even if the length of bus lines are reduced by 10\% from the current service. Thus, the service providers can improve their operational efficiency by consolidating the bus routes by considering MPKST links. On the other hand, for the number of OD pairs shown in \figref{fig:number_of_od_pairs}, MPKST is smaller than the existing bus service. As shown in \figref{fig:actual_network}, there are nodes other than hubs which are arranged like leaves in a tree, and the path travel distance inevitably becomes longer for some nodes. Since the major aim of this study is the finding of hub nodes and trunk routes, the response to the fairness issue should be updated in the line setting in TND-STA. In this case, efficient route planning can be expected by considering a line setting such as a feeder bus starting from a hub discovered by the TND-STA.

\begin{figure}[htp]
\begin{tabular}{cc}
      \begin{minipage}[t]{0.45\columnwidth}
        \centering
        \includegraphics[keepaspectratio, scale=0.4]{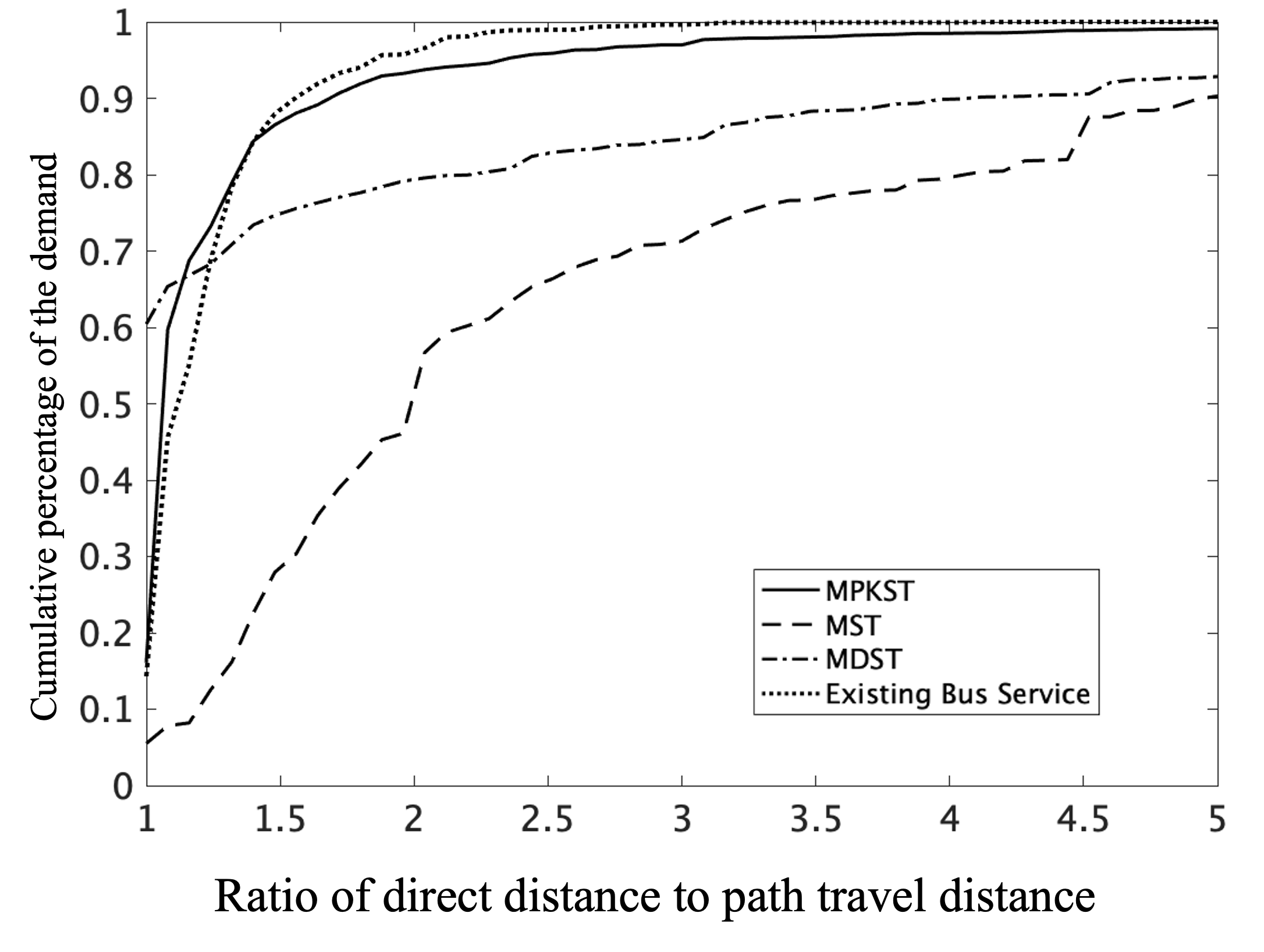}
        \subcaption{The volume of demand}
        \label{fig:cumulative_frequency_demand}
      \end{minipage} &
      \begin{minipage}[t]{0.4\columnwidth}
        \centering
        \includegraphics[keepaspectratio, scale=0.4]{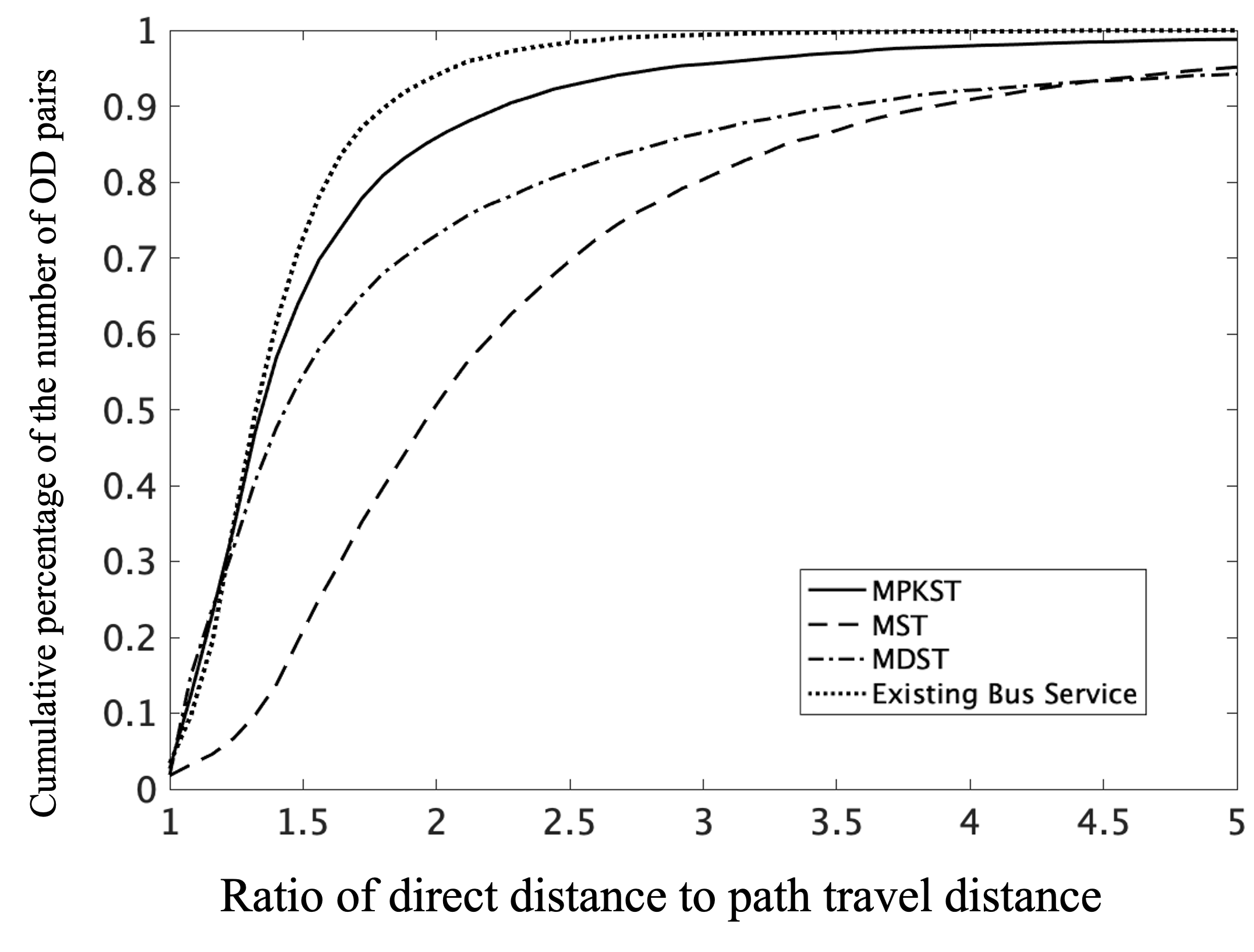}
        \subcaption{The number of OD pairs}
        \label{fig:number_of_od_pairs}
      \end{minipage}
    \end{tabular}
     \caption{Relationship between cumulative frequency distribution and ratio of link distance to path travel distance.}\label{fig:relationship_between_cumulative_frequency_and_ratio}
\end{figure}

\begin{figure}[htp]
  \centering
  \includegraphics[keepaspectratio, scale=0.4]{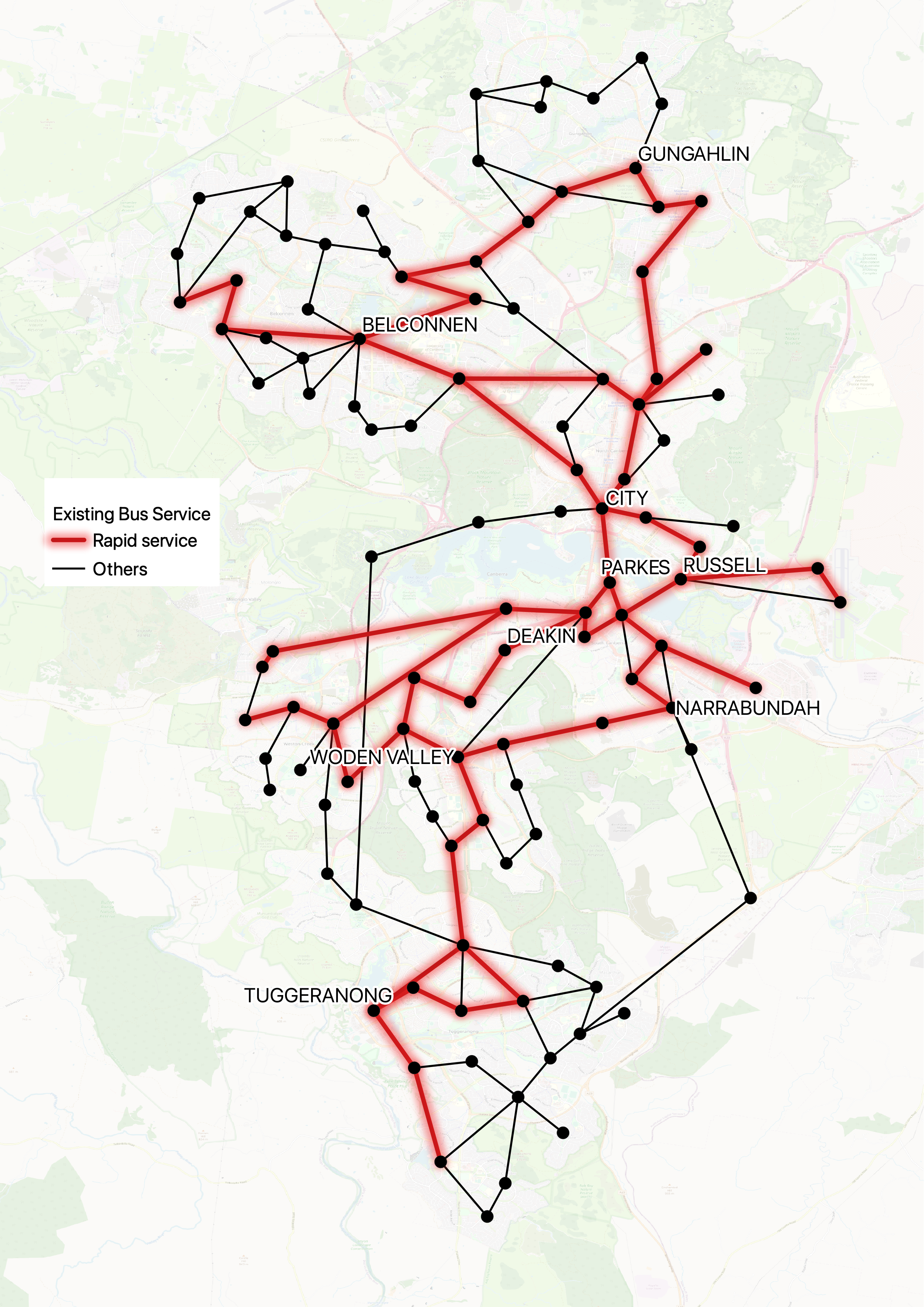}
  \caption{Topology of Canberra public transport network topology.}
  \label{fig:actual_network}
\end{figure}

\section{Further improvements on transit network design}

In a transit network design problem (TNDP), the spanning tree constraint is a very strong constraint on network topology. This constraint does not allow closed paths that can lead to significant detours. The application to the Canberra bus network shows that many OD pairs or demand measured in passenger-kilometers do not deviate significantly from the link distance, but smaller passenger-kilometers are possible if closed paths are allowed. If the policy maker needs to gain a better network topology than MPKST, they might want to relax the number of links constraint and add more links. Let's add a link $y_{\left(ij\right)}$ to MPKST that minimizes passenger-kilometers, and call it graph $T_{+1}^{*}$. Further, add the link $y_{ij}$ that minimizes passenger-kilometers to $T_{+1}^{*}$ and call it $T_{+2}^{*}$. This iterative process is considered to be a greedy algorithm with additional links to minimize passenger-kilometers with MPKST as the initial state. The network after adding $\alpha$ links is represented by the incremental equation as follows:

\begin{equation}
  T_{+\alpha}^{*}=T_{+\alpha-1}^{*}\cup{\argmin_{y_{\left(ij\right)}}{Z\left(T_{+\alpha-1}^{*}\cup y_{\left(ij\right)}\right)}}
    \label{eq:greedy_algorithm_with_additional_links}
\end{equation}

This additional work does not diminish the superiority of the proposed TND-STA and the proposed tabu search, because problems with a constraint on the number of links not equal to $N-1$ are difficult to solve without a fast solution method like our method. The greedy algorithm with MPKST as initial value is easy to solve.

The changes in passenger-kilometers with sequential relaxation of the number of links constraint are shown in \figref{fig:greedy_algorithm_with_additional_links}, where $\mathbf{td}$ is the sum of the product of the link distance and the demand, and is the lower bound of TNDP when the number of links constraint is removed. Adding a small number of links to MPKST greatly reduces passenger-kilometers, while the marginal effects become smaller when more links are added. \figref{fig:greedy_algorithm_cumulative_frequency} shows the results of the same analysis as \figref{fig:relationship_between_cumulative_frequency_and_ratio} on a graph with 10 links added to MPKST by the greedy algorithm of \eqref{eq:greedy_algorithm_with_additional_links}. $T_{+10}^{*}$ shows a significant improvement in cumulative percentage with only 10 additional links compared to MPKST. Particularly, more than 90\% of passengers fall into the range of $\frac{c_{\left(ij\right)}}{t_{\left(ij\right)}}\leq 1.5$, which indicates that $T_{+10}^{*}$ significantly increases the cumulative percentage compared to the existing bus service. The proposed extension approach can also help policy makers in finding trunk routes and hub nodes.

\begin{figure}[htp]
  \centering
  \includegraphics[keepaspectratio, scale=0.5]{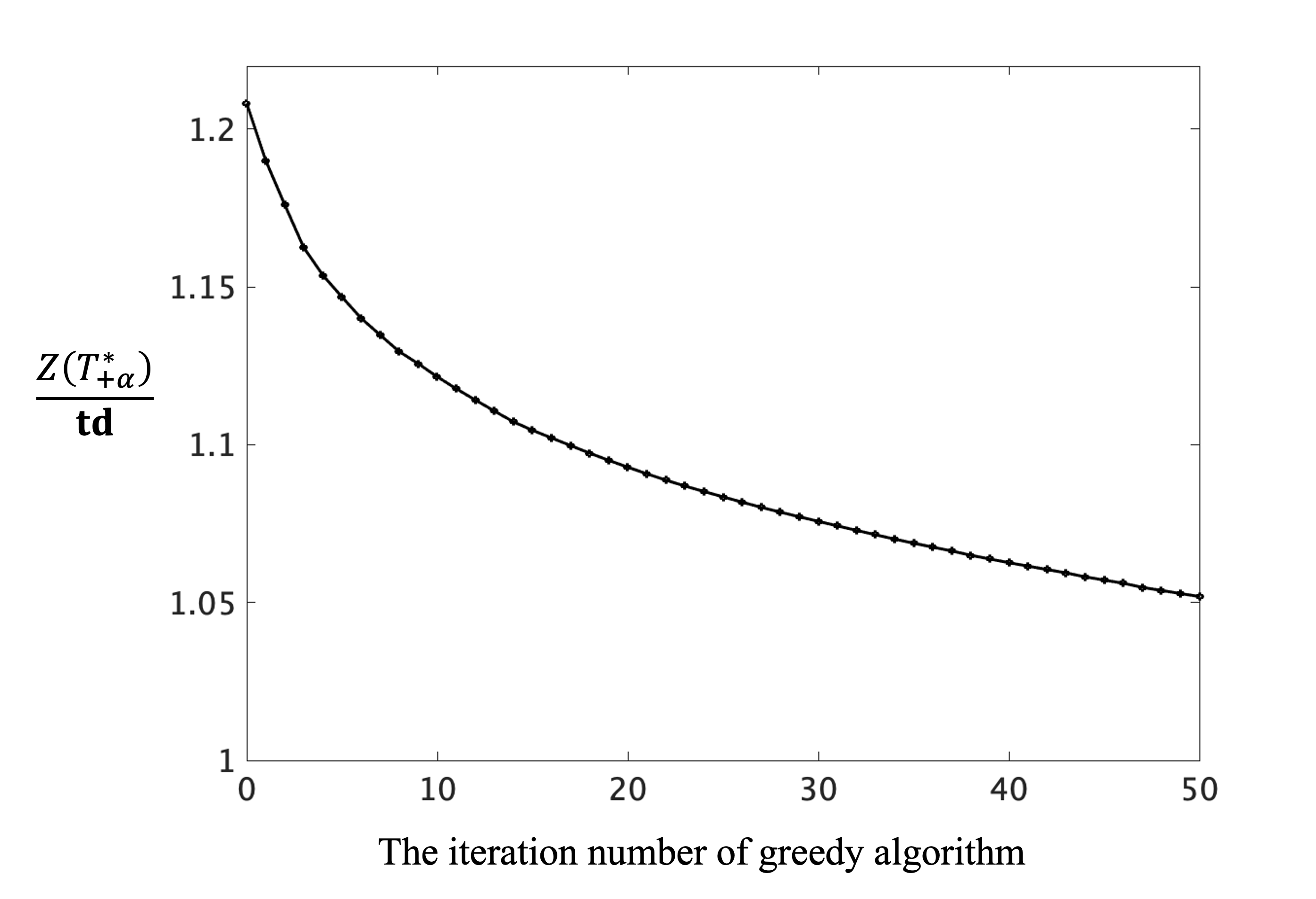}
  \caption{Transition of the ratio of the lower bound to the passenger-kilometers obtained using the greedy algorithm.}
  \label{fig:greedy_algorithm_with_additional_links}
\end{figure}

\begin{figure}[htp]
  \centering
  \includegraphics[keepaspectratio, scale=0.6]{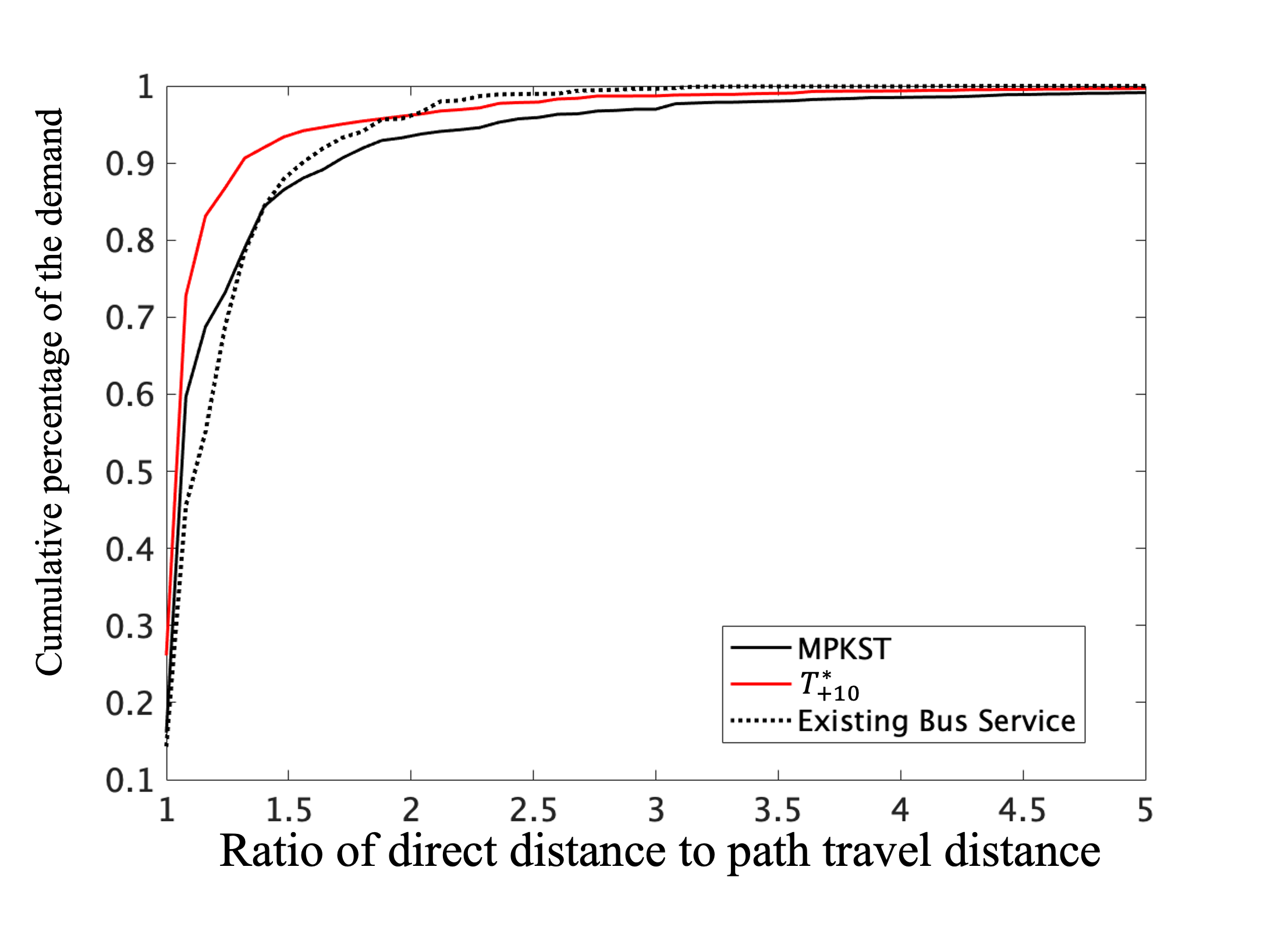}
  \caption{Relationship between cumulative frequency distribution of the demand and ratio of link distance to path travel distance. The red line is the result for the graph with 10 links added by the greedy algorithm from MPKST.}
  \label{fig:greedy_algorithm_cumulative_frequency}
\end{figure}

\section{Conclusion}
This study proposes an optimization model to study a transit network design problem using the concept of spanning tree, which minimizes the total passenger-kilometers in the network. Also, we develop a solution algorithm: Link Swapping with Tabu Search to quickly solve the model. Unlike the state-of-the-art heuristic algorithms in \citet{Bell2020} that carry out many shortest path calculations, this tabu search algorithm only needs to carry out a few shortest path calculations, improving the computational performance. The numerical experiments using the Canberra bus network data demonstrates the superiority of the proposed algorithm in terms of computation time and solution quality, compared with the Link Swapping and Link Deleting heuristic algorithms in \citet{Bell2020}.

The MPKST obtained from the proposed tabu search algorithm shows the major hubs and the trunk routes in the Canberra bus network, which are consistent with the current bus network in Canberra (shown in \figref{fig:canberra_rapid_network}). We confirmed that the proposed method can transport a large number of travel demands in a shorter path travel time with fewer links than the existing bus service. Thus, the proposed methodology is worth for preliminary planning in the bus network (re-)design. We also provided a method to further reduce passenger-kilometers using a greedy algorithm with the tree derived by the TND-STA as the initial state. This method relaxes the restriction on the number of links from the tree derived by TND-STA, but provides a more effective connection structure for preliminal planning.

The primary aim in this paper is to help design the topology of a transit network without consideration of frequency or service design. Having identified hubs and trunk routes, services can be designed (routes, frequencies, vehicle type, etc.) in a second step. Also, the proposed methodology can help reduce number and length of roads used in route design because of the way a spanning tree concentrates flows on a limited number of links. This will also increase service frequencies at stops thereby reducing passenger waiting times. Moreover, the areas with large demand are likely to be more directly connected capturing major rapid lines in the solution. The limitation of the current research is that it only considers demand and topology. Driver assignment, equity of user, etc. are also important in transit network design. Also, verification of the practicality of the proposed methodology in other cities is needed. We will leave these two directions for our next study.

\bibliographystyle{plainnat}
\bibliography{SPT}

\end{document}